\pdfoutput=1\relax
\documentclass[reqno]{amsart}

\usepackage{verbatim}
\usepackage[textsize=scriptsize]{todonotes}
\usepackage{tikz-cd}
\usepackage{etoolbox}
\usepackage{etex}
\usepackage[hyperfootnotes=false]{hyperref}
\usepackage[multiple]{footmisc}
\usepackage{cleveref}
\usepackage[T1]{fontenc}

\usepackage{chemarr}
\usepackage{amssymb}
\usepackage{amsmath}
\usepackage{comment}
\usepackage{cleveref}
\usepackage{mathtools}
\usepackage{rotating}
\usepackage{wrapfig}
\usepackage{outlines}
\usepackage{graphicx}

\theoremstyle{definition}
\newtheorem{nul}{}[section]
\newtheorem{dfn}[nul]{Definition}

\newtheorem{rmk}[nul]{Remark}

\newtheorem{cnstr}[nul]{Construction}
\newtheorem{cnv}[nul]{Convention}
\newtheorem{ntn}[nul]{Notation}
\newtheorem{exm}[nul]{Example}

\newtheorem{rec}[nul]{Recollection}

\newtheorem{qst}[nul]{Question}
\newtheorem{prb}[nul]{Problem}
\newtheorem*{dfn*}{Definition}
\newtheorem*{axm*}{Axiom}
\newtheorem*{ntn*}{Notation}
\newtheorem*{exm*}{Example}
\newtheorem*{exr*}{Exercise}
\newtheorem*{int*}{Intuition}
\newtheorem*{qst*}{Question}
\newtheorem*{rmk*}{Remark}

\theoremstyle{plain}

\newtheorem{thm}[nul]{Theorem}
\newtheorem{prop}[nul]{Proposition}
\newtheorem{lem}[nul]{Lemma}

\newtheorem{cor}[nul]{Corollary}

\newtheorem*{thm*}{Theorem}
\newtheorem*{prop*}{Proposition}
\newtheorem*{cor*}{Corollary}
\newtheorem*{lem*}{Lemma}
\newtheorem*{cnj*}{Conjecture}

\let\oldwidetilde\widetilde
\protected\def\widetilde{\oldwidetilde}

\DeclareMathOperator{\Map}{\mathrm{Map}}

\DeclareMathOperator{\Hom}{\mathrm{Hom}}

\DeclareMathOperator{\Ss}{\mathbb{S}}

\DeclareMathOperator{\F}{\mathbb{F}}
\DeclareMathOperator{\QQ}{\mathbb{Q}}

\DeclareMathOperator{\C}{\mathcal{C}}

\DeclareMathOperator{\CP}{\mathbb{CP}}
\DeclareMathOperator{\E}{\mathbb{E}}

\DeclareMathOperator{\MU}{\mathrm{MU}}

\DeclareMathOperator{\Spec}{\text{Spec}}

\DeclareMathOperator{\pic}{\mathrm{pic}}
\DeclareMathOperator{\Pic}{\mathrm{Pic}}

\DeclareMathOperator{\Ext}{\mathrm{Ext}}

\newcommand{\wt}{\widetilde}

\newcommand{\MSpin}{\mathrm{MSpin}}
\newcommand{\MString}{\mathrm{MString}}

\newcommand{\tmf}{\mathrm{tmf}}

\newcommand{\CC}{\mathbb{C}}

\newcommand{\Z}{\mathbb{Z}}

\def\Einf{\mathbb{E}_\infty}
\def\Hinf{\mathbb{H}_\infty}
\def\Spf{\mathrm{Spf}}
\def\MX{\mathrm{MX}}
\def\BC{\mathrm{BC}}
\def\BU{\mathrm{BU}}
\def\H{\mathrm{H}}
\def\Or{\mathrm{Or}}
\def\OO{\mathcal{O}}
\def\TMF{\mathrm{TMF}}
\def\Tmf{\mathrm{Tmf}}
\def\tmf{\mathrm{tmf}}
\def\M{\mathcal{M}}
\def\m{\mathfrak{m}}

\def\L{\mathcal{L}}

\def\R{\mathbb{R}}
\def\G{\mathbb{G}}
\def\Th{\mathrm{Th}}
\def\sm{\mathrm{sm}}
\def\Guniv{\mathbb{G}^{\mathrm{univ}}}
\def\univ{\mathrm{univ}}
\def\top{\mathrm{top}}
\def\cusp{\mathrm{cusp}}
\def\KU{\mathrm{KU}}
\def\bu{\mathrm{bu}}
\def\ku{\mathrm{ku}}
\def\gl{\mathrm{gl}}

\def\bgl{\mathrm{bgl}}
\def\bsl{\mathrm{bsl}}
\def\MUP{\mathrm{MUP}}
\def\Gal{\mathrm{Gal}}

\DeclarePairedDelimiter\abs{\lvert}{\rvert}%
\makeatletter
\let\oldabs\abs
\def\abs{\@ifstar{\oldabs}{\oldabs*}}

\usepackage{tikz}
\usetikzlibrary{matrix,arrows,decorations}
\usepackage{tikz-cd}

\usepackage{adjustbox}

\let\oldtocsection=\tocsection
 
\let\oldtocsubsection=\tocsubsection
 
\let\oldtocsubsubsection=\tocsubsubsection
 
\renewcommand{\tocsection}[2]{\hspace{0em}\oldtocsection{#1}{#2}}
\renewcommand{\tocsubsection}[2]{\hspace{1em}\oldtocsubsection{#1}{#2}}
\renewcommand{\tocsubsubsection}[2]{\hspace{2em}\oldtocsubsubsection{#1}{#2}}

\usepackage{etoolbox}
\newtoggle{draft}
\togglefalse{draft}

\iftoggle{draft} {
\usepackage[margin=1.5in]{geometry}
\newcommand{\NB}[1]{\todo[color=gray!40]{#1}}
\newcommand{\TODO}[1]{\todo[color=red]{#1}}
}{ 
\usepackage[margin=1.5in]{geometry}
\newcommand{\NB}[1]{}
\newcommand{\TODO}[1]{}
\renewcommand{\todo}[1]{}
\renewcommand{\todo}[1]{}
}
\title{Obstruction theory and the level $n$ elliptic genus}

\author{Andrew Senger}
\address{Department of Mathematics, Harvard University, Cambridge, MA, USA}
\email{senger@math.harvard.edu}

\begin{document}
\begin{abstract}
  Given a height $\leq 2$ Landweber exact $\Einf$-ring $E$ whose homotopy is concentrated in even degrees, we show that any complex orientation of $E$ which satisfies the Ando criterion admits a unique lift to an $\Einf$-complex orientation $\MU \to E$.
  As a consequence, we give a short proof that the level $n$ elliptic genus lifts uniquely to an $\Einf$-complex orientation $\MU \to \tmf_1 (n)$ for all $n \geq 2$.
\end{abstract}
\maketitle

\setcounter{tocdepth}{1}
\tableofcontents
\vbadness 5000


\section{Introduction} \label{sec:intro}
Complex-oriented ring spectra play a central role in the chromatic approach to homotopy theory.
Given a homotopy associative ring spectrum $E$, recall that a complex orientation is a choice of class $u \in \wt{E}^2 (\CP^\infty)$ with the property that its restriction along $S^2 \cong \CP^1 \hookrightarrow \CP^\infty$ is the unit $1 \in E^0 (\ast) \cong \wt{E}^{2} (S^2)$.
A complex orientation determines an isomorphism of graded rings
\[E^* (\CP^\infty) \cong E^* [\![u]\!].\]

Complex orientations may also be described in terms of the complex cobordism spectrum $\MU$:  complex orientations of $E$ are in natural bijection with maps of homotopy associative ring spectra $\MU \to E$.

The complex cobordism spectrum $\MU$ admits much more structure than that of a homotopy associative ring spectrum: it is an $\Einf$-ring spectrum.
When $E$ also admits the structure of an $\Einf$-ring spectrum, it is natural to ask whether a given complex orientation is induced by a map of $\Einf$-ring spectra
\[\MU \to E.\]
We will call such maps $\Einf$-complex orientations.
The $\Einf$-complex orientations of an $\Einf$-ring may be viewed as particularly canonical complex orientations.

Indeed, many of the most familiar (not necessarily complex) orientations admit lifts to $\Einf$-orientations.
For example, the Atiyah--Bott--Shapiro orientations \cite{ABS}
\[\MSpin \to \mathrm{ko} \hspace{1cm} \text{and} \hspace{1cm} \MSpin^{\mathbb{C}} \to \mathrm{ku}.\]
were refined to $\Einf$-orientations by Joachim \cite{Joachim}, who gave an explicit geometric construction of such an $\Einf$-orientation.
Indeed, one expects that any orientation of geometric origin may, with enough care, be refined to an $\Einf$-orientation.
A more sophisticated example is the Ando--Hopkins--Rezk $\Einf$-String orientation of the connective $\Einf$-ring of topological modular forms $\tmf$ \cite{AHR}
\[\MString \to \tmf,\]
which refines the Witten genus \cite{Witten87,Witten86}.
While it is expected that this $\Einf$-orientation has a geometric origin, and much work has gone into developing such a viewpoint (for a small sampling, see \cite{ST1, ST2, Costello1, Costello2, Nets, BE}), such a description has so far remained elusive.
In \Cref{thm:main-elliptic} below, we will prove that the Hirzebruch level $n$ elliptic genera \cite{Hirz}\cite[\S 5]{Witten86} for $n \geq 2$ may be lifted to $\Einf$-complex orientations
\[\MU \to \tmf_1 (n).\]
%
%
%
\subsection{The Ando criterion}

An algebraic approximation of what it means for a complex orientation $\MU \to E$ to be $\Einf$ is given by the \textit{Ando criterion}.
It asks that the complex orientation be compatible in a suitable sense with the power operations in $E$.
In many cases, the Ando criterion is equivalent to the property that that the complex orientation $\MU \to E$ be a map of $\Hinf$-ring spectra.


  Now let $E$ denote an $\Einf$, or more generally an $\Hinf$-ring spectrum, with a fixed complex orientation $u$.
  \begin{ntn}
    Given a complex vector bundle $V \to X$ of dimension $d$, we let $t_u (V) \in E^{2d} (\Th(V))$ denote the Thom class of $V$.
  \end{ntn}

  Fix a prime $p$, let $\rho$ denote the vector bundle over $\BC_p$ corresponding to the complex regular representation of $C_p$, and let $\gamma_1$ denote the tautological bundle over $\CP^\infty$.
  Let $I_{tr} \subset E^* (\BC_p)$ denote the transfer ideal.
  Recall from \cite[\S 7]{HL} that there are additive power operations
  \[\Psi_u : E^{2*} (\Th(\gamma_1)) \to E^{2p*} (\Th(\rho \boxtimes \gamma_1))/I_{tr}.\]

\begin{dfn}
  We say that a complex orientation of $E$ satisfies the Ando criterion at the prime $p$ if
  \[\Psi_u (t_u (\gamma_1)) = t_u (\rho \boxtimes \gamma_1)\]
  in $E^{2p} (\Th(\rho \boxtimes \gamma_1))/I_{tr}$.
  We say that a complex orientation of $E$ satisfies the Ando criterion if it satisfies the Ando criterion for all primes $p$.
\end{dfn}

\begin{rmk}
  If $E$ is $p$-local, then $E^* (\BC_\ell) / I_{tr} = 0$ for all primes $\ell \neq p$.
  Therefore a complex orientation of a $p$-local $\Einf$-ring satisfies the Ando criterion if and only if it satisfies the Ando criterion at $p$.
\end{rmk}

The complex cobordism spectrum $\MU$, equipped with the canonical complex orientation, satisfies the Ando criterion.
It follows that any $\Einf$-complex orientation, or more generally any $\Hinf$-complex orientation, satisfies the Ando criterion.

\subsection{Results}

The first main theorem of this paper states that for many $\Einf$-rings of height $\leq 2$, this condition is sufficient, and that the resulting $\Einf$-complex orientations are determined up to homotopy by their underlying complex orientations:

\begin{thm} \label{thm:main-obst}
  Let $E$ denote a height $\leq 2$ Landweber exact $\Einf$-ring whose homotopy is concentrated in even degrees. Then any complex orientation $\MU \to E$ which satisfies the Ando criterion lifts uniquely up to homotopy to an $\Einf$-ring map $\MU \to E$:
\end{thm}

In the above theorem, we say that a Landweber exact ring spectrum $E$ is of height $\leq n$ if $v_n \in \pi_* (E) / (p, v_1, \dots, v_{n-1})$ is a unit for all primes $p$.\footnote{It may happen that $\pi_* (E) / (p, v_1, \dots, v_{n-1})$ is the zero ring. In this case, our convention is to regard the unique element of the zero ring as a unit.}
As a corollary, we obtain the following result for height $\leq 2$ Lubin--Tate theories:

\begin{cor}
  Let $k \subseteq \overline{\F}_p$ denote a field of characteristic $p > 0$ which is algebraic over the prime field $\F_p$, and let $\G$ denote a formal group of height $\leq 2$ over $k$.
  Then any complex orientation of the associated $2$-periodic Morava $K$-theory $K(k,\G)$ lifts uniquely up to homotopy to an $\Einf$-complex orientation of $E(k, \G)$
  \[\MU \to E(k, \G).\]
\end{cor}

\begin{proof}
  This follows immediately from \Cref{thm:main-obst} and \cite[Theorem 1.2]{Zhu}, which implies that every complex orientation of $K(k,\G)$ admits a unique lift to a complex orientation of $E(k,\G)$ which satisfies the Ando criterion whenever $k$ is algebraic over $\F_p$.
\end{proof}

\begin{rmk}
  \Cref{thm:main-obst} was inspired by recent work of Balderrama \cite[Theorem 6.5.3]{Bald}. Balderrama shows that every periodic complex orientation of a Lubin--Tate theory of height $\leq 2$ satisfying an analogue of the Ando criterion lifts to an $\Einf$-orientation.
  He also shows that $\Einf$-refinements of periodic complex orientations of even periodic $K(1)$-local $\Einf$-rings exist whenever the Ando criterion is satisfies and are unique up to homotopy.

  The key observation he makes is the presence of evenness in the Goerss--Hopkins obstruction theory for (periodic) $\Einf$-complex orientations, which implies that the obstructions to existence and uniqueness appearing in his theorem vanish for formal reasons.
  Our results about existence of $\Einf$-complex orientations will be obtained by observing a similar evenness in the Hopkins--Lawson obstruction theory \cite{HL}.

  In contrast to \Cref{thm:main-obst}, Balderrama does not obtain any uniqueness results at height $2$ \cite[Remark 6.5.4]{Bald}. In \Cref{rmk:periodic-unique}, we will prove that $\Einf$-refinements of periodic complex orientations of height $2$ Lubin--Tate theories are unique.
\end{rmk}

Our second main theorem uses \Cref{thm:main-obst} to give a simple proof of the following theorem:

\begin{thm} \label{thm:main-elliptic}
  For $n \geq 2$, the Hirzebruch level $n$ elliptic genus lifts uniquely up to homotopy to a map of $\Einf$-rings
  \[\MU \to \tmf_1 (n).\]
\end{thm}

While \Cref{thm:main-obst} does not apply directly to $\tmf_1 (n)$, it does apply to $\TMF_1 (n)$, and it is not hard to upgrade the resulting $\Einf$-complex orientation to one of $\tmf_1 (n)$.
This reduces us to verifying the Ando criterion, which may be done following the strategy of Ando--Hopkins--Strickland \cite{AHS}.

\begin{rmk}
  During the writing process, \cite{Absmeier} has also appeared, which follows the strategy of Ando--Hopkins--Rezk \cite{AHR} to prove a similar result to \Cref{thm:main-elliptic} for $\Tmf_1 (n)$.
  Our method is rather different from that of Absmeier and completely avoids the consideration of $p$-adic Eisenstein measures.
\end{rmk}

\subsection{Further questions}

One of key inputs in our proof of \Cref{thm:main-obst} is \Cref{thm:morava-K-even}, which states that the Morava $K$-theory of certain finite groups is concentrated in even degrees. This is closely related to these groups being \textit{good} in the sense of Hopkins--Kuhn--Ravenel \cite[\S 7]{HKR}.

These groups come in a family, and to see that the Hopkins--Lawson obstruction theory is concentrated in even degrees one would like to show that the entire family has Morava $K$-theory concentrated in even degrees, cf. \Cref{rmk:is-even}:

\begin{qst}
  Is the Morava $K$-theory of the groups $\Gamma_k ^{(n)}$ of \Cref{dfn:gammakn} concentrated in even degrees for $k > 2$, at least for $n \gg 0$?
  The groups $\Gamma_k ^{(1)}$ are the extraspecial $p$-groups of type $p^{1+2k} _+$.
\end{qst}

One could also ask about $C_2$-equivariant refinement of our results.

\begin{qst}
  The complex cobordism spectrum may be refined to a $C_2$-equivariant $\Einf$-ring $\MU_\R$, and $\tmf_1 (n)$ admits the natural structure of a $C_2$-equivariant $\Einf$-ring \cite[Theorem 2.20]{Meierleveln}.
  Moreover, the Hirzebruch level $n$ elliptic genus admits a refinement to a map of homotopy $C_2$-ring spectra $\MU_\R \to \tmf_1 (n)$ \cite[Theorem 3.5]{Meierleveln}.
  Is there a suitable $C_2$-equivariant analogue of \Cref{thm:main-obst} which may be used to prove a $C_2$-equivariant refinement of \Cref{thm:main-elliptic}?
  See \cite[Remark 13]{HL} for a comment on a $C_2$-equivariant version of the Hopkins--Lawson obstruction theory.
\end{qst}

On the other hand, it would be very interesting to study $\Einf$-complex orientations at heights $\geq 3$.
A natural choice of spectra to study would be Lubin--Tate spectra.
Since the obstruction groups don't vanish for formal reasons at these heights, it seems likely that this this will require an explicit analysis of the Goerss--Hopkins obstruction theory for $\Einf$-maps $\MU \to E(k,\G)$.
In particular, one would have to compute the $\mathrm{E}_2$-page.

\begin{prb} \label{prop:obst}
  Compute the Goerss--Hopkins obstruction groups for $\Einf$-maps $\MU \to E(k,\G)$ for height $\geq 3$ Lubin--Tate theories $E(k,\G)$.
\end{prb}
%

\subsection{Acknowledgements}
The author would like to thank Robert Burklund and Jeremy Hahn for useful comments on a draft.
During the course of this work, the author was supported by NSF Grants DGE-1745302 and DMS-2103236.

\section{Existence of $\Einf$-Orientations} \label{sec:obst}
In this section, we will prove the half of \Cref{thm:main-obst} concerning the existence of $\Einf$-complex orientations.
The main tool that we will utilize is an obstruction theory for $\Einf$-complex orientations studied by Hopkins--Lawson \cite{HL}.
Given this obstruction theory, the existence half of \Cref{thm:main-obst} reduces to the statement that certain obstruction groups vanish.
Using work of Arone--Lesh \cite{ALFiltered}, this can be further reduced to the evenness of the Morava $K$-theory of certain extraspecial $p$-groups and related groups, which we are able to extract from the literature.
%
%
\subsection{Hopkins--Lawson obstruction theory}
We begin by summarizing the main results of the Hopkins--Lawson obstruction theory \cite{HL}.
First, a definition:

\begin{dfn}
  Let $E$ denote a homotopy commutative ring.
  We let $\Or(E)$ denote the space of complex orientations of $E$, i.e. the fiber
  \[\Or(E) \to \Map (\Sigma^{\infty-2} \CP^\infty, E) \to \Map(\Sigma^{\infty-2} \CP^2, E) \simeq \Map(\Ss, E)\]
  above the unit map $\Ss \to E$.
\end{dfn}

\begin{thm} [{\cite[Theorems 1 \& 32]{HL}}] \label{thm:HL}
  There exists a diagram of $\mathbb{E}_\infty$-ring spectra
  \[\Ss \to \MX_1 \to \MX_2 \to \MX_3 \to \dots \to \MU\]
  such that the following hold:
  \begin{enumerate}
    \item The natural map $\varinjlim \MX_n \to \MU$ is an equivalence.
    \item The $\Einf$-ring $\MX_1$ is equipped with a natural complex orientation inducing an equivalence $\Map_{\Einf} (\MX_1, E) \xrightarrow{\sim} \Or(E)$ for each $\Einf$-ring $E$.
    \item Given $m > 0$ and an $\Einf$-ring $E$, there is a pullback square
      \begin{center}
        \begin{tikzcd}
          \Map_{\mathbb{E}_\infty} (\MX_m, E) \ar[r] \ar[d] & \Map_{\mathbb{E}_\infty} (\MX_{m-1} , E) \ar[d] \\
          \{\ast\} \ar[r] & \Map_* (F_m, \Pic(E)),
        \end{tikzcd}
      \end{center}
      where $F_m$ is a certain pointed space described below.
    \item The map $\MX_{m-1} \to \MX_m$ is a rational equivalence if $m > 1$, a $p$-local equivalence if $m$ is not a power of $p$, and a $K(n)$-local equivalence if $m > p^n$.
    \item Let $E$ denote an $\Einf$-ring such that $\pi_* E$ is $p$-local and $p$-torsion free. Then an $\Einf$-map $\MX_1 \to E$ extends to an $\Einf$-map $\MX_p \to E$ if and only if the corresponding complex orientation of $E$ satisfies the Ando criterion.
  \end{enumerate}
\end{thm}

Using this theorem, we will reduce the proof of \Cref{thm:main-obst} to the following:

\begin{lem} \label{lem:even}
  Let $E$ denote a $p$-complete Landweber exact ring spectrum with homotopy concentrated in even degrees.
  Then
  $E^{2n} (F_p) \cong E^{2n+1} (F_{p^2}) \cong 0$
  for all $n \in \Z$.
%
\end{lem}

\begin{rmk}
  In fact, we only need that $E^{2n+1} (F_{p^2}) \cong 0$.
  However, we include the statement $E^{2n} (F_p) \cong 0$ since it is no harder for us to prove.
  This extra evenness implies uniqueness up to homotopy for $\Einf$-refinements of complex orientations of height $\leq 1$.
  However, we will prove uniqueness in a different way in \Cref{sec:unique}.
\end{rmk}

\begin{qst} \label{qst:higher-hts}
  Given a $p$-complete Landweber exact ring spectrum $E$ with homotopy concentrated in even degrees, is $E^{2*+k-1} (F_{p^k}) \cong 0$ for $k \geq 3$?
\end{qst}

\begin{proof}[Proof of existence in \Cref{thm:main-obst} assuming \Cref{lem:even}]
  We begin by reducing to the case where $E$ is $p$-complete for some prime $p$.
  We make use of the fracture square:
  \begin{center}
    \begin{tikzcd}
      E \ar[r]\ar[d] & \prod_p E^{\wedge} _p \ar[d] \\
      E_{\QQ} \ar[r] & \left( \prod_p E^{\wedge} _p \right)_{\QQ}.
    \end{tikzcd}
  \end{center}
  Note that $E^{\wedge} _p$ is again even and Landweber exact, cf. \Cref{sec:app}.

  By \Cref{thm:HL}, we have $\Map_{\Einf} (\MU, R) \simeq \Or(R)$ for a rational $\Einf$-ring $R$; it follows further that $\pi_1 \Map_{\Einf} (\MU, R) \cong \pi_1 \Or(R) \cong 0$ if $R$ has homotopy concentrated in even degrees.
  As a consequence, there are pullback squares of sets:
  \begin{center}
  \vspace{-0.5cm}
    \begin{equation} \label{eq:pullback}
    \begin{tikzcd}
      \pi_0 \Map_{\Einf} (\MU, E) \ar[r] \ar[d] & \pi_0 \Map_{\Einf} (\MU, \prod_p E^{\wedge} _p) \ar[d] \\
      \pi_0 \Or (E_{\QQ}) \ar[r] & \pi_0 \Or ((\prod_p E^{\wedge} _p)_{\QQ})
    \end{tikzcd}
  \end{equation}
  \end{center}
  and
  \begin{center}
    \begin{tikzcd}
      \pi_0 \Or (E) \ar[r] \ar[d] & \pi_0 \Or (\prod_p E^{\wedge} _p) \ar[d] \\
      \pi_0 \Or (E_{\QQ}) \ar[r] & \pi_0 \Or ((\prod_p E^{\wedge} _p)_{\QQ}).
    \end{tikzcd}
  \end{center}
  To lift a complex orientation of $E$ to an $\Einf$-complex orientation, it therefore suffices to lift the induced complex orientation of $E^{\wedge} _p$.
  We may therefore assume that $E$ is $p$-complete.

  Let $E$ now denote an $p$-complete Landweber exact $\Einf$-ring with homotopy concentrated in even degrees. Using \Cref{thm:HL}, we see that it suffices to show that
  \[\pi_0 \Map_{\Einf} (\MX_{p^2}, E) \to \pi_0 \Map_{\Einf} (\MX_p, E)\]
  is surjective.


  By \Cref{thm:HL}, there is a fiber sequence
  \[ \Map_{\Einf} (\MX_{p^2}, E) \to \Map_{\Einf} (\MX_p, E) \to \Map_* (F_{p^2}, \Pic(E))\]
Now, there are equivalences
\[\Map_* (F_m, \Pic(E)) \simeq \Hom (\Sigma^{\infty} F_m, \pic(E)) \simeq \Hom (\Sigma^{\infty} F_m, \Sigma E),\]
where in the second equivalence we have used the fact that $\Sigma^\infty F_m$ is $(2m-1)$-connected by \cite[Corollary 4(5)]{HL}.
  It therefore follows from the above fiber sequence that it suffices to show that
  \[E^1 (\Sigma^{\infty} F_{p^2}) \cong 0.\]
  Since $E$ is $p$-complete, this follows from \Cref{lem:even}.
%
%
%
\end{proof}

%
%

\subsection{Proof of \Cref{lem:even}}
In the remainder of this section, we will prove \Cref{lem:even}.
First, we must recall the definition of the spaces $F_m$.

\begin{rec}
  Let $L_m$ denote the nerve of the topologized poset of proper direct-sum decompositions of $\CC^m$, and let $(L_m) ^{\diamond}$ denote its unreduced suspension. The natural action of $U(m)$ on $\CC^m$ endows $L_m$ and $(L_m)^{\diamond}$ with the structure of $U(m)$-spaces.

  Furthermore, view $S^{2m}$ as a $U(m)$-space via its identification with the one-point compactification of $\CC^m$, viewed as the fundamental representation of $U(m)$.
Then $F_m$ is given by
\[F_m \simeq ((L_m)^{\diamond} \wedge S^{2m})_{h U(m)}.\]
\end{rec}

To prove \Cref{lem:even}, we will use a result of Arone--Lesh to reduce the study of the $E$-cohomology of $F_m$ to the $E$-cohomology of certain groups $\Gamma_k$, whose definition we now recall.

\begin{dfn}[{\cite[Definition 1]{pStubborn}}] \label{dfn:gammakn}
  Let $\sigma_0, \dots, \sigma_{k-1} \in \Sigma_{p^k}$ denote the permutations
  \[\sigma_r (i) = \begin{cases} i + p^r &\text{ if } i \equiv 1, \dots, (p-1)p^r \mod p^{r+1} \\ i-(p-1) p^r &\text{ if } i \equiv (p-1)p^r +1, \dots, p^{r+1} \mod p^{r+1}. \end{cases}\]
    We let $\Gamma_k \subset U(p^k)$ denote the subgroup generated by the permutation matrices corresponding to $\sigma_0, \dots, \sigma_{k-1}$, the central $S^1$, and the diagonal matrices $A_0, \dots, A_{k-1}$ given by
  \[(A_r)_{ii} = \zeta_p ^{\lfloor (i-1)/p^r \rfloor},\]
  where $\zeta_p$ is a primitive $p$th root of unity.
  Then $\Gamma_k$ lies in a central extension
  \[1 \to S^1 \to \Gamma_k \to \F_p ^{2k} \to 1.\]
  For each $n \geq 1$, there is a normal subgroup $\Gamma_k ^{(n)} \subset \Gamma_k$ which only contains the central $p^{n}$th roots of unity instead of the full $S^1$.
  Then there are central extensions
  \[1 \to C_{p^n} \to \Gamma_k ^{(n)} \to \F_p^{2k} \to 1\]
  and exact sequences
  \[1 \to \Gamma_k ^{(n)} \to \Gamma_k \to S^1 \to 1.\]
\end{dfn}

\begin{rmk}
  The groups $\Gamma_k ^{(1)}$ are examples of \textit{extraspecial $p$-groups}, and in this language are commonly denoted by $p^{1+2k} _+$.
\end{rmk}

\begin{prop} [{\cite[Propositions 9.6 \& 10.3]{ALFiltered}}] \label{prop:summand}
  The $p$-completion of the spectrum
  \[\Sigma^\infty F_m \simeq \Sigma^{\infty} ((L_m)^{\diamond} \wedge S^{2m})_{hU(m)}\]
  is null unless $m = p^k$, in which case it is a summand of the $p$-completion of
  \[\Sigma^{k} (S^{2p^k})_{h\Gamma_k},\]
  where $\Gamma_k$ acts on $S^{2p^k}$ via the inclusion $\Gamma_k \subset U(p^k)$.
\end{prop}

Now, $(S^{2p^k})_{h\Gamma_k}$ may also be described as the Thom spectrum associated to the composition
\[B\Gamma_k \to \BU (p^k) \to \Z \times \BU,\]
from which it follows that
\[E^* ((S^{2p^k})_{h\Gamma_k}) \cong \widetilde{E}^{*-2p^k} (B\Gamma_k).\]

\Cref{lem:even} therefore reduces to the following lemma:

\begin{lem}\label{lem:gamma-k-even}
  Let $E$ denote a $p$-local Landweber exact ring spectrum whose homotopy is concentrated in even degrees.
  Then, for $k \leq 2$, $E^* (B\Gamma_k)$ is concentrated in even degrees.
\end{lem}

We will deduce \Cref{lem:gamma-k-even} from the following Morava $K$-theory computations:

\begin{thm} \label{thm:morava-K-even}
  For all $n \geq 0$, the following groups are concentrated in even degrees:
  \begin{enumerate}
    \item (Tezuka--Yagita, \cite[Theorem 4.2]{TY}) $K(n)^* (B\Gamma_1 ^{(1)})$ at all primes $p$.
    \item (Schuster--Yagita, \cite[Theorem 5.4]{SY}) $K(n)^* (B\Gamma_2 ^{(1)})$ at the prime $2$.
    \item (Yagita, \cite[Theorem 1.2]{Y}) $K(n)^* (B\Gamma_2 ^{(2)})$ at odd primes $p$.
  \end{enumerate}
\end{thm}

A lemma of Strickland which builds on work of Ravenel--Wilson--Yagita \cite{RWY} allows us to transport this evenness from Morava $K$-theory to $E$-cohomology:

\begin{cor}\label{cor:E-even}
  Given a $p$-local Landweber exact ring spectrum $E$ whose homotopy is concentrated in even degrees, the following groups are concentrated in even degrees:
  \begin{enumerate}
    \item $E^* (B\Gamma_1 ^{(1)})$ at all primes $p$.
    \item $E^* (B\Gamma_2 ^{(1)})$ at the prime $2$.
    \item $E^* (B\Gamma_2 ^{(2)})$ at odd primes $p$.
  \end{enumerate}
\end{cor}

\begin{proof}
  Combine \Cref{thm:morava-K-even} with \cite[Lemma 8.25]{FSFG}.\footnote{Note that we cannot apply \cite[Lemma 8.25]{FSFG} directly to $B\Gamma_k$, since it only applies to spaces with bounded above $\QQ$-cohomology.}\footnote{While it is assumed in \cite{FSFG} that $E$ is even-periodic, this is not used in the proof of \cite[Lemma 8.25]{FSFG}. We also note that $E$ may be viewed as a summand of an even-periodic ring spectrum $E[x_2 ^{\pm 1}]$, so that we may conclude from \cite[Lemma 8.25]{FSFG} as stated.}
\end{proof}

\begin{proof}[Proof of \Cref{lem:gamma-k-even}]
  The short exact sequence
  \[\Gamma_k ^{(n)} \to \Gamma_k \to S^1\]
  induces a fiber sequence
  \[B\Gamma_k ^{(n)} \to B\Gamma_k \to BS^1.\]
  Using the associated Atiyah--Hirzebruch spectral sequence
  \[\H^* (BS^1; E^* (B\Gamma_k ^{(n)})) \Rightarrow E^* (B\Gamma_k),\]
  we find that if $E^* (B\Gamma_k ^{(n)})$ is even for some $n$, then $E^* (B\Gamma_k)$ must be as well.
  Therefore \Cref{lem:gamma-k-even} follows from \Cref{cor:E-even}.
\end{proof}

\begin{rmk} \label{rmk:is-even}
  By the same arguments, to give a positive answer to \Cref{qst:higher-hts} it suffices to show that $K(n)^* (B\Gamma_k ^{(i_k)})$ is concentrated in even degrees for a fixed $i_k$ not depending on $n$.
  This is closely related to the question of whether $\Gamma_k ^{(i_k)}$ is a \textit{good group} in the sense of Hopkins--Kuhn--Ravenel \cite[\S 7]{HKR}.
\end{rmk}
%

\section{Uniqueness of $\Einf$-Orientations} \label{sec:unique}
Our goal in this section is to prove the following result:

\begin{thm} \label{thm:unique}
  Let $E$ denote an $L_2$-local complex orientable $\Einf$-ring with the property that $K(1)_* E$ and $K(1)_* L_{K(2)} E$ are concentrated in even degrees at all primes $p$.
  Then each complex orientation of $E$ admits at most one refinement to an $\Einf$-complex orientation up to homotopy.
\end{thm}

\begin{exm}\label{exm:land-unique}
  Any Landweber exact ring spectrum $E$ of height $\leq 2$ whose homotopy is concentrated in even degrees satisfies the hypotheses of \Cref{thm:unique}.
  By \Cref{lem:Kn-Land}, $L_{K(2)} E$ is again Landweber exact and has homotopy concentrated in even degrees.
  It therefore suffices to show that $K(1)_* E$ is concentrated in even degrees.
  This is true because $K(1)_* E \cong K(1)_* \MU \otimes_{\pi_* \MU} \pi_* E$, and $K(1)_* \MU$ is concentrated in even degrees.
\end{exm}

Combining \Cref{thm:unique} with \Cref{exm:land-unique}, we obtain the uniqueness half of \Cref{thm:main-obst}.

\begin{exm}\label{exm:unique-tmf}
  The ring spectra $\Tmf_1 (n)$ satisfy the hypotheses of \Cref{thm:unique}.
  This follows from \cite[Propositions 2.4 and 2.6]{DylanTmf}, which imply that the $p$-complete complex $K$-theory of these spectra and their $K(2)$-localizations is $p$-torsionfree and concentrated in even degrees.
\end{exm}

Our proof of \Cref{thm:unique} will be based on the orientation theory of \cite{ABGHR} and the following lemma:

\begin{lem} \label{lem:torfree-even}
  Let $E$ denote an $\MU$-module with the property that $K(1)_* E$ is concentrated in even degrees.
  Then $[\KU_p, L_{K(1)} E]$ is torsionfree and $[\Sigma \KU_p, L_{K(1)} E] = 0$.
\end{lem}

%

\begin{dfn}
  We say that a $(\KU_p)_*$-module is \textit{pro-free} if it is the $p$-completion of a free module.
  Moreover, given a spectrum $X$, we write $\KU^{\vee}_* (X)$ for $\pi_* L_{K(1)} (\KU \otimes X)$.
\end{dfn}

\begin{proof}
  Since $K(1)_* (\KU_p)$ is even, $\KU^{\vee} _* (\KU_p)$ is pro-free by \cite[Proposition 8.4(f)]{HoveyStrickland}.
  Therefore, by \cite[Proposition 1.14]{BarthelHeard}, there is an isomorphism
  \[\pi_* \Hom (\KU_p, L_{K(1)} (\KU_p \otimes E) \cong \Hom_{(\KU_p)_*} (\KU^{\vee}_* (\KU_p), \KU^{\vee}_* (E)).\]
  By assumption, $K(1)_* (E)$ is even and hence $\KU^{\vee} _* (E)$ is even and pro-free by \cite[Proposition 8.4(f)]{HoveyStrickland}. In particular, it is torsionfree, so that $\Hom_{(\KU_p)_*} (\KU^{\vee} _* (\KU_p), \KU^{\vee} _* (E))$ is even and torsionfree.

  The result then follows from the following facts:
  \begin{enumerate}
    \item Since $E$ is an $\MU$-module, the unit map $L_{K(1)} E \to L_{K(1)} (\MU \otimes E)$ admits a splitting given by the module structure map.
    \item There is a splitting of the map $L_{K(1)} \MU \to L_{K(1)} E(1)$ which is compatible with the unit map \cite[Theorem 4.1]{HoveySadofsky}.
    \item There is an equivalence of spectra $\KU_p \simeq \bigoplus_{i=0} ^{p-2} \Sigma^{2i} L_{K(1)} E(1)$. \qedhere
  \end{enumerate}
\end{proof}

\begin{proof}[Proof of \Cref{thm:unique}]
  Using the pullback square of sets (\ref{eq:pullback}), we may assume that $E$ is $p$-complete.

  Recall the map $\Sigma^{\infty} \CP^{\infty} \to \bu$ which is adjoint to the canonical map $\CP^\infty \to \BU$.
  By orientation theory \cite{ABGHR}, we must show that
\[[\bu, \gl_1 (E)] \to [\Sigma^{\infty} \CP^{\infty}, \gl_1 (E)]\]
  is injective.
  Since $E$ is $p$-complete, $\gl_1 (E)$ agrees with $\gl_1 (E)^{\wedge} _p$ in degrees $\geq 2$, so we may as well replace the former by the latter.
  Letting $F$ denote the fiber of the map $\gl_1 (E) \to L_2 \gl_1 (E)$,
  we find that there is a fiber sequence
  \[F^{\wedge} _p \to \gl_1 (E)^{\wedge}_p \to L_{K(1) \oplus K(2)} \gl_1 (E).\]
  It therefore suffices to show that
  \[[\bu, F^{\wedge}_p] \to [\Sigma^{\infty} \CP^{\infty}, F^{\wedge}_p]\]
  and
  \[[\bu, L_{K(1) \oplus K(2)} \gl_1 (E)] \to [\Sigma^{\infty} \CP^{\infty}, L_{K(1) \oplus K(2)} \gl_1 (E)]\]
  are injective.
  The first is injective because $F^{\wedge} _p$ is $3$-coconnective by \cite[Theorem 4.11]{AHR} and the cofiber of $\Sigma^{\infty} \CP^{\infty} \to \bu$ is $4$-connective.

  To prove that the second is injective, we first note that the Bousfield--Kuhn functor \cite{Bousfield,Kuhn} and the chromatic fracture square for $L_{K(1) \oplus K(2)} \gl_1 (E)$ imply that there is an exact sequence
  \[[\Sigma \KU_p, L_{K(1)} L_{K(2)} E] \to [\bu, L_{K(1) \oplus K(2)} E] \to [\KU_p, L_{K(1)} E].\]
  Applying \Cref{lem:torfree-even}, we learn that $[\bu, L_{K(2) \oplus K(1)} E]$ is torsionfree.
  It therefore injects into its rationalization, so that the result follows from the fact that $\Sigma^{\infty} \CP^{\infty} \to \bu$ is a rational equivalence.
\end{proof}

\begin{rmk} \label{rmk:periodic-unique}
  Balderrama has shown that \textit{periodic} complex orientations of height $2$ Lubin--Tate theories $E(k,\G)$ which satisfy a version of the Ando criterion admit lifts to periodic $\Einf$-complex orientations $\MUP \to E(k,\G)$ \cite[Theorem 6.5.3(3)]{Bald}.
  In this remark, we prove that such lifts are unique up to homotopy.

  We begin with a result which has been proven by Rezk at height $2$ (which is the case that we use) \cite{RezkHL} and in general is an unpublished theorem of Hopkins and Lurie. Let $\overline{k}$ denote the algebraic closure of $k$. Then we have:
  \begin{align*}
    \pi_* \mathbb{G}_m (E(\overline{k}, \G)) \coloneqq \pi_* \Map (\Z, \gl_1 (E(\overline{k}, \G))) \cong \begin{cases} \overline{k}^{\times} &*=0 \\ \Z_p &*=3 \\ 0 &\text{otherwise}. \end{cases}
  \end{align*}
  On $\pi_0$, the map $\Z \to \gl_1 (E(\overline{k}, \G))$ corresponding to $a \in \overline{k}^{\times}$ picks out the Teichmuller lift $[a]$.

  Since $\Gal(k)$ has $p$-cohomological dimension at most $1$ \cite[Proposition 6.1.9]{Galois}, it follows that
  \[\pi_0 \G_m (E(k, \G)) \cong k^{\times}.\]
  In particular, the map $\pi_0 \G_m (E(k,\G)) \to \pi_0 \gl_1 (E(k,\G))$ is injective.
  Now, by orientation theory it suffices to show that
  \[[\ku, \gl_1 (E(k,\G))] \to [\Sigma^\infty _+ \CP^\infty, \gl_1 (E(k,\G))]\]
  is injective.
  By what we have proven above about uniqueness of $\Einf$-complex orientations, it suffices to show that
  \[[\Z, \gl_1 (E(k,\G))] \to [\Ss^0, \gl_1 (E(k, \G))]\]
  is injective, which is what we showed above.

  As noted in \cite[Remark 6.5.4]{Bald}, this implies that the Goerss--Hopkins obstruction group $\Ext^2_{\Delta} (\hat{Q}(E(k,\G)_0 ^{\wedge} \MUP), \omega)$ is equal to $0$. 
\end{rmk}

\section{The level $n$ elliptic genus} \label{sec:genus}
\begin{cnv}
  In this section $n$ will denote an integer greater than or equal to $2$.
\end{cnv}

In this section, we will prove \Cref{thm:main-elliptic}, which states that the Hirzebruch level $n$ elliptic genus lifts uniquely up to homotopy to an $\Einf$-complex orientation $\MU \to \tmf_1 (n)$.

We will begin by recalling some background material about $\Einf$-rings of topological modular forms with level-$\Gamma_1 (n)$ structures in \Cref{sec:tmf-level}.
In \Cref{sec:theta}, we recall from \cite{AHSCube, AHS} how complex orientations may be described in terms of $\Theta^1$-structures.
We then describe the level $n$ elliptic genus in this language, following Meier \cite[\S 3]{Meierleveln}.
In \Cref{sec:reduction}, we show that a complex orientation for $\TMF_1 (n)$ satisfies the Ando criterion at $p$ if and only if its composition along a map $\TMF_1 (n) \to E(k, \G)$ to a Lubin--Tate theory does.

In \Cref{sec:Ando}, we recall from \cite{AHS} how the Ando criterion for Lubin--Tate theories may be rephrased in terms of $\Theta^1$-structures.
We then prove that the Hirzebruch level $n$ genus satisfies the Ando criterion.
As a consequence of \Cref{thm:main-obst}, it lifts uniquely up to homotopy to an $\Einf$-ring map $\MU \to \TMF_1 (n)$.
Finally, in \Cref{sec:lift}, we prove that this $\Einf$-ring map admits a unique up to homotopy lift to an $\Einf$-ring map $\MU \to \tmf_1 (n)$, completing the proof of \Cref{thm:main-elliptic}.

\subsection{Topological modular forms with level structures}\label{sec:tmf-level}

In this section, we will recall some basic facts about the $\Einf$-rings $\tmf_1 (n)$, $\Tmf_1 (n)$ and $\TMF_1 (n)$ of topological modular forms with level-$\Gamma_1 (n)$ structure.
We begin by recalling the algebraic background.

\begin{dfn}
  We let $\M_1 (n)$ denote the Deligne--Mumford moduli stack of elliptic curves with level-$\Gamma_1 (n)$ structure over $\Z[\frac{1}{n}]$.
  Concretely, given a scheme $S$ on which $n$ is invertible, we have
  \[\M_1 (n) (S) = \text{elliptic curves } E \text{ over } S \text{ with a point } P \in E[n] (S) \text{ of exact order }n.\]

  Moreover, we write $\overline{\M_1} (n)$ for the Deligne--Rapoport moduli stack of generalized elliptic curves with level-$\Gamma_1 (n)$ structure \cite{DR}. This is again a Deligne--Mumford stack over $\Z[\frac{1}{n}]$, and $\M_1 (n) \subset \overline{\M_1} (n)$ is an open substack.

  We denote the universal family of curves by $\pi : \C \to \overline{\M_1} (n)$, and write $\pi^{\sm} : \C^{\sm} \to \overline{\M_1} (n)$ for the smooth locus. Then $\C^{\sm}$ admits admits the natural structure of a group scheme, and we write $\omega$ for the line bundle of invariant differentials.

  Finally, we write $\M = \M_1 (1)$ and $\overline{\M} = \overline{\M_1} (1)$ for the moduli stacks of (generalized) elliptic curves without level structure.
\end{dfn}

\begin{rec}[{Goerss--Hopkins--Miller \cite[Chapter 12]{tmfBook}, Lurie \cite{ECII}, Hill--Lawson \cite{HLtmf}}]
  There is a sheaf $\OO^{\top}_{\overline{\M}}$ of $\E_\infty$-ring spectra on the Kummer log-\'etale site of $\overline{\M}$ with the following properties:
  \begin{enumerate}
    \item There is a natural isomorphism of sheaves $\pi_0 \OO^{\top}_{\overline{\M}} \cong \OO_{\overline{\M}}$.
    \item There are natural isomorphisms of quasicoherent sheaves $\pi_{2i} \OO^{\top} _{\overline{\M}} \cong \omega^{i}$ and $\pi_{2i+1} \OO^{\top}_{\overline{\M}} \cong 0$.
    \item Write $\widehat{\pi^{\sm}} : \widehat{\C^{\sm}} \to \overline{\M}$ for the completion of $\pi^{\sm} : \C^{\sm} \to \overline{\M}$ along the zero section.
      There is a natural isomorphism of sheaves of rings $\pi_0 \Hom (\Sigma^\infty _+ \CP^{\infty}, \OO^{\top}) \cong (\widehat{\pi^{\sm}})_* \OO_{\widehat{\C^{\sm}}}$.
  \end{enumerate}

  Since the natural morphisms $\overline{\M_1} (n) \to \overline{\M}$ are Kummer log-\'etale (we refer the reader to \cite{HLtmf} for more details), we may define $\Einf$-rings:
  \[\TMF_1 (n) \coloneqq \Gamma(\M_1 (n), \OO^{\top})\]
  and
  \[\Tmf_1 (n) \coloneqq \Gamma(\overline{\M_1 (n)}, \OO^{\top}).\]
\end{rec}

By definition, there are spectral sequences
\[\H^s (\M_1 (n), \omega^i) \Rightarrow \pi_{2i-s} \TMF_1 (n)\]
and
\[\H^s (\overline{\M_1} (n), \omega^i) \Rightarrow \pi_{2i-s} \Tmf_1 (n).\]

\begin{rmk}\label{rmk:tmf1n-nice}
  Let $n \geq 2$. It follows from \cite[Proposition 2.4(4)]{MeierAddDec} that $\H^s (\M_1 (n), \omega^i) = 0$ for all $s > 0$, so that $\TMF_1 (n)$ has homotopy groups concentrated in even degrees and that there are natural isomorphisms
\[\pi_{2i} (\TMF_1 (n)) \cong \Gamma(\M_1 (n), \omega^i).\]

  Moreover, $\TMF_1 (n)$ is Landweber exact. Indeed, this is a consequence of flatness of $\M_1 (n)$ over $\Z[\frac{1}{n}]$, the integrality of $\M_1 (n)_{\F_p}$ and the fact that the formal group of an elliptic curve is of height $\leq 2$.
\end{rmk}

However, because the groups $\H^1 (\overline{\M_1}, \omega^i)$ do not in general vanish, we do not have a similar theorem for $\Tmf_1 (n)$.
Instead, we have the $\Einf$-ring spectrum $\tmf_1 (n)$ from \cite{Meierleveln}.

\begin{rec}[{\cite[Theorem 1.1]{Meierleveln}}]
  There is an essentially unique connective $\Einf$-ring spectrum $\tmf_1 (n)$ whose homotopy groups are concentrated in even degrees and which is equipped with an $\Einf$-ring map
  \[\tmf_1 (n) \to \Tmf_1 (n)\]
  such that the induced maps
  \[\pi_{2i} \tmf_1 (n) \to \pi_{2i} \Tmf_1 (n) \to \Gamma(\overline{\M_1} (n), \omega^{i})\]
  are isomorphisms.
\end{rec}

\begin{rmk}
  There is a sequence of natural maps
  \[\tmf_1 (n) \to \Tmf_1 (n) \to \TMF_1 (n).\]
\end{rmk}

\subsection{$\Theta^1$-structures}\label{sec:theta}
In this section, we describe complex orientations in terms of $\Theta^1$-structures and give a description of the Hirzebruch level $n$ genus in this language. 

Suppose that we are given a base Deligne--Mumford stack $S$ and a formal group or generalized elliptic curve $G$ over $S$. We denote the structure map by $p : G \to S$ and the zero section by $0 : S \to G$.
Given a line bundle $\L$ on $G$, we set 
\[\Theta^1 (\L) \coloneqq p^* 0^* \L \otimes \L^{-1}.\]
There is a canonical trivialization $0^* \Theta^1 (\L) \cong \OO_{S}$.

\begin{dfn}
  A $\Theta^1$-structure on a line bundle $\L$ over $G$ is a trivialization of $\Theta^1 (\L)$ which pulls back to the canonical trivialization of $0^* \Theta^1 (\L)$.

  Equivalently, a $\Theta^1$-structure on $\L$ is an isomorphism $p^* 0^* \L \cong \L$ which pulls back to the canonical isomorphism $0^* p^* 0^* \L \cong 0^* \L$.
\end{dfn}

\begin{rmk} \label{rmk:theta-inv}
  It is clear from the definition that there is a natural bijection between $\Theta^1$-structures on $\L$ and $\L^{-1}$.
\end{rmk}

\begin{rmk}
  Note that a line bundle $\L$ on $G$ admits a $\Theta^1$-structure if and only if it is pulled back from a line bundle on $S$ via the structure map $p$.
\end{rmk}

When $G$ is a generalized elliptic curve, $\Theta^1$-structures are unique when they exist:

\begin{lem} \label{lem:ell-theta-unique}
  Suppose that $G$ is a generalized elliptic curve. Then the set of $\Theta^1$-structures on a line bundle $\L$ over $G$ is either empty or consists of a single element.
\end{lem}

\begin{proof}
  The set of isomorphisms $p^* 0^* \L \cong \L$ is a torsor for $\OO_G (G)^{\times}$, whereas the set of isomorphisms $0^* p^* 0^* \L \cong 0^* \L$ is a torsor for $\OO_S (S)^{\times}$.
  It therefore suffices to show that the map $\OO_G (G) \to \OO_S (S)$ induced by pulling back along the zero section is an isomorphism.
  This follows from \cite[\href{https://stacks.math.columbia.edu/tag/0E0L}{Tag 0E0L}]{stacks-project}.
\end{proof}

Now let $E$ denote an even weakly periodic homotopy commutative ring spectrum, and let $\G_E = \Spf E^0 (\CP^\infty)$ denote the associated formal group over $\Spec \pi_0 E$. We denote by $\gamma_1$ the canonical line bundle over $\CP^\infty$ and let $\OO_{\G_E} (1)$ denote the line bundle over $\G_E$ corresponding to
\[E^0 (\Th(\gamma_1)) \cong \widetilde{E^0} (\CP^\infty) \cong \ker (\OO_{\G_E} \xrightarrow{0^*} \pi_0 E).\]
Then we have the following theorem:

\begin{prop}[{{\cite[Theorem 2.48]{AHSCube}}}] \label{prop:Theta1-orient}
  There is a natural bijection between complex orientations of $E$ and $\Theta^1$-structures on $\OO_{\G_E} (1)$.
\end{prop}

\begin{proof}
  Recall that a complex orientation of $E$ consists of an element of $\widetilde{E^2} (\CP^{\infty})$ which restricts to the unit along the map $\widetilde{E^2} (\CP^\infty) \to \widetilde{E^2} (\CP^1) \cong \pi_0 E$.
  Since the map $E^0 (\CP^\infty) \to \pi_0 E$ is an infinitesimal thickening, complex orientations of $E$ may be identified with $E^0 (\CP^\infty)$-module isomorphisms $E^0 (\CP^{\infty}) \cong \widetilde{E^2} (\CP^{\infty})$ that become equal to the canonical isomorphism $\widetilde{E^2} (\CP^1) \cong \pi_0 E$ after tensoring down along $E^0 (\CP^\infty) \to \pi_0 E$.

  By \Cref{rmk:theta-inv}, we may replace $\OO_{\G_E} (1)$ with $\OO_{\G_E} (1)^{-1}$ in the statement of the proposition.
  To prove the proposition, it therefore suffices to identify the global sections of $\Theta^1 (\OO_{\G_E} (1)^{-1})$ with $\widetilde{E^2} (\CP^{\infty})$ and pullback along the zero section with $\widetilde{E^2} (\CP^\infty) \to \widetilde{E^2} (\CP^1) \cong \pi_0 E$.
  This is an immediate consequence of the definitions.
%
\end{proof}

Even though $\tmf_1 (n)$ is not weakly even periodic, Meier has shown that its complex orientations may still be described in terms of $\Theta^1$-structures.

\begin{prop}[{\cite[Lemma 3.2]{Meierleveln}}] \label{prop:tmf-theta1}
  There is a natural bijection between complex orientations of $\tmf_1 (n)$ and $\Theta^1$-structures on $\OO_{\widehat{\C}} (1)$ over $\widehat{C} \to \overline{\M_1} (n)$.
\end{prop}

We now recall the treatment of Hirzebruch's level $n$ elliptic genus from \cite[\S 3]{Meierleveln}. The following proposition is a mild rephrasing of \cite[Lemma 3.3]{Meierleveln}.

\begin{prop}
  Let $P : \overline{\M_1} (n) \to \C^{\sm}$ denote the universal level-$\Gamma_1 (n)$ structure.
  \begin{enumerate}
    \item The pullback of the line bundle $\OO_{\C} ([0] - [P])$ on $\C$ to $\widehat{\C}$ is naturally isomorphic to $\OO_{\widehat{\C}} (1)$.
    \item There is a degree $n$ \'etale cover $q : \C' \to \C$ of generalized elliptic curves so that $q^* \OO_{\C} ([0] - [P])$ admits a (necessarily unique) $\Theta^1$-structure.
  \end{enumerate}
\end{prop}

Since the induced map $\widehat{q} : \widehat{\C'} \xrightarrow{\sim} \widehat{\C}$ is an isomorphism, we obtain a $\Theta^1$-structure on $\OO_{\widehat{C}} (1)$, and hence by \Cref{prop:tmf-theta1} a complex orientation of $\tmf_1 (n)$.
This is the complex orientation corresponding to the \textit{Hirzebruch level $n$ elliptic genus}.

\subsection{Reduction to Lubin--Tate theory}\label{sec:reduction}
Our goal in this section is to prove \Cref{prop:reduction-LT} below.
This proposition implies that to verify the Ando criterion for the Hirzebruch level $n$ elliptic genus, it suffices to verify the Ando criterion after composition with a map $\TMF_1 (n) \to E(k, \G)$ to a Lubin--Tate theory.
This will be useful to us because work of Ando--Hopkins--Strickland \cite{AHS} rephrases the Ando criterion for Lubin--Tate theories in terms of $\Theta^1$-structures.

We begin by recalling some basic facts about about Lubin--Tate theories.

\begin{rec}[{Goerss--Hopkins \cite{GH}, Lurie \cite{ECII}}]
  Let $k$ denote a perfect field of characteristic $p > 0$ and let $\G$ denote a formal group of finite height $h$ over $k$.
  To the pair $(k, \G)$ one may associate an $\Einf$-ring spectrum $E(k,\G)$, known as the \textit{Lubin--Tate spectrum} of $(k,\G)$.
  The ring $\pi_0 E(k,\G)$ is naturally isomorphic to the universal deformation ring of $(k,\G)$, which is noncanonically isomorphic to $\mathbb{W}(k) [\![u_1, \dots, u_{h-1}]\!]$.
  We let $\Guniv$ denote the universal deformation of $\G$ over $\pi_0 E(k, \G)$, and denote by $\omega_{\Guniv}$ its module of invariant differentials.
  There are natural isomorphisms of $\pi_{2i} E(k,\G)$ with $\omega^{i} _{\Guniv}$.
\end{rec}

\begin{cnstr}
  Let $k$ denote a perfect field of characteristic $p > 0$. Associated to a supersingular $k$-point $(E, \alpha \in E[n] (k))$ of $\M_1 (n)$, there is a map of $\Einf$-ring spectra $\TMF_1 (n) \to E(k, \widehat{E})$. Indeed, this follows from the description of $E(k, \widehat{E})$ as an oriented deformation ring \cite[\S 6]{ECII}, Lurie's Serre--Tate theorem for strict abelian varieties \cite[\S 7]{ECI}, and the universal property of $\TMF_1 (n)$.
\end{cnstr}

\begin{prop} \label{prop:reduction-LT}
  Let $k$ denote a perfect field of characteristic $p$ and let $(E, \alpha \in E[n](k))$ denote an object of $\M_1 (n)(k)$ with $E$ a supersingular elliptic curve. Then a complex orientation of $\TMF_1 (n)$ satisfies the Ando criterion at $p$ if and only if its composite with the canonical map $\TMF_1 (n) \to E(k, \widehat{E})$ does.
\end{prop}

Our proof of \Cref{prop:reduction-LT} rests on the following lemma:

\begin{lem} \label{lem:BCp-inj}
  Let $E \to F$ denote a map of complex orientable homotopy commutative ring spectra.
  Suppose that $\pi_* E$ and $\pi_* F$ are $p$-torsionfree and that the induced map $\pi_* E / p \to \pi_* F / p$ is an injection.
  Then the map $E^* (\BC_p)/ I_{tr} \to F^* (\BC_p)/ I_{tr}$ is an injection.
\end{lem}

\begin{proof}
   Since $\pi_* E$ is $p$-torsionfree, a choice of complex orientation gives rise to an isomorphism (see \cite[Proposition 4.2]{QuillenElem}):
  \[E^* (\BC_p)/I_{tr} \cong E^* [[t]] / \langle p \rangle (t)\]
  where $\abs{t} = 2$, $\langle p \rangle (t) = \frac{[p](t)}{t}$ and $[p] (t)$ is the $p$-series of the formal group on $E^*$. Moreover, since $\pi_* E$ is $p$-torsionfree, we may identify $t^{n-1} E^* [[t]] / (\langle p \rangle (t), t^n)$ with a shift of $E^* / p$. The analogous statements for $F$ also hold.

  Taking the induced complex orientation of $F$, we identify $E^* (\BC_p) / I_{tr} \to F^* (\BC_p) / I_{tr}$ with the natural map $E^* [[t]] / \langle p \rangle (t) \to F^* [[t]] / \langle p \rangle (t)$.
  To show that this map is injective, it suffices to show that it is injection on the associated graded for the $t$-adic filtration.
  This identifies with a shift of the natural map $E^* / p \to F^* / p$ in each degree, which is an injection by hypothesis.
\end{proof}

\begin{proof}[Proof of \Cref{prop:reduction-LT}]
  It is clear that if a complex orientation of $\TMF_1 (n)$ satisfies the Ando criterion at $p$, so does the induced complex orientation of $E(k, \widehat{E})$.
  To prove the converse, it suffices to show that the induced map
  \[\TMF_1 (n)^* (\BC_p)/I_{tr} \to E(k, \widehat{E})^* (\BC_p)/I_{tr}\]
  is an injection.
  We begin by reducing to the case $n \geq 5$.
  Any map $\Spec k \to \M_1 (n)$ fits into a diagram
  \begin{center}
    \begin{tikzcd}
      \Spec k' \ar[r] \ar[d] & \M_1 (n^2) \ar[d] \\
      \Spec k \ar[r] & \M_1 (n)
    \end{tikzcd}
  \end{center}
  for $k'$ a finite separable extension of $k$.
  It follows that there is a diagram
  \begin{center}
    \begin{tikzcd}
      \TMF_1 (n) \ar[r] \ar[d] & E(k, \widehat{E}) \ar[d] \\
      \TMF_1 (n^2) \ar[r] & E(k', \widehat{E}_{k'}),
    \end{tikzcd}
  \end{center}
  so that it suffices to show that \[\TMF_1 (n)^* (\BC_p) / I_{tr} \to \TMF_1 (n^2)^* (\BC_p) / I_{tr}\] and \[\TMF_1 (n^2)^* (\BC_p) / I_{tr} \to E(k', \widehat{E}_{k'})^* (\BC_p) / I_{tr}\] are injective.

  We may therefore assume that $n \geq 5$ if we can show that the first map is an injection.
  Both $\pi_* \TMF_1 (n)$ and $\pi_* \TMF_1 (n^2)$ are $p$-torsionfree, and
  \[\pi_{2i} \TMF_1 (n) / p \cong \Gamma(M_1 (n)_{\F_p}, \omega^{i}) \to \Gamma(M_1 (n^2)_{\F_p}, \omega^{i}) \cong \pi_{2i} \TMF_1 (n) / p\]
  is an injection since $\M_1 (n^2)_{\F_p} \to \M_1 (n)_{\F_p}$ is a finite \'etale cover.
  Therefore we may apply \Cref{lem:BCp-inj} to show that the first map above is an injection.

  We may now assume that $n \geq 5$, so that $\M_1 (n)$ is represented by an affine scheme $\Spec R_n$ \cite[Proposition 2.4(2)]{MeierAddDec}.
  Let $\m \subset R_n$ denote the kernel of the map $R_n \to k$. Then $R_n / \m$ is a finite field (hence perfect), and the pair $(E, P \in E[n](k))$ descends to $R_n/ \m$. As a consequence, there is a factorization $\TMF_1 (n) \to E(R_n/ \m, \widehat{E}) \to E(k, \widehat{E})$.
  It follows immediately from \Cref{lem:BCp-inj} that the induced map
  \[E(R_n/ \m, \widehat{E})^* (\BC_p) / I_{tr} \to E(k, \widehat{E})^*  (\BC_p)/ I_{tr}\]
  is an injection, so that it suffices to show that
  \[\TMF_1 (n)^* (\BC_p) / I_{tr} \to E(R_n/ \m, \widehat{E})^* (\BC_p)/ I_{tr}\]
  is an injection.
  To apply \Cref{lem:BCp-inj}, we need to show that $\pi_{*} \TMF_1 (n) / p \to \pi_{*} E(R_n / \m, \widehat{E}) / p$ is an injection.
  By abuse of notation, we let $\omega$ denote the invertible $R_n$-module corresponding to the line bundle $\omega$ on $\M_1 (n)$.
  Then the above map can be identified with the $\m$-adic completion map $(\omega^{*/2} / p) \to (\omega^{*/2} / p)^{\wedge} _\m$.
  This is an injection by the Krull intersection theorem, since $\M_1 (n)_{\F_p}$ is an integral scheme.
%
\end{proof}

\subsection{The Ando criterion and $\Theta^1$-structures}\label{sec:Ando}

In this section, we recall from the work of Ando--Hopkins--Strickland how the Ando criterion for Lubin--Tate theories may be rephrased in terms of $\Theta^1$-structures.
We refer the reader to \cite{AHS} for proofs and further details.
We then combine this rephrasing with \Cref{prop:reduction-LT} to prove that the Hirzebruch level $n$ elliptic genus satisfies the Ando criterion.
Finally, we deduce from \Cref{thm:main-obst} that the Hirzebruch level $n$ elliptic genus lifts uniquely up to homotopy to an $\Einf$-ring map
\[\MU \to \TMF_1 (n).\]

We begin by recalling from \cite[\S 14]{AHS} how $\Theta^1$-structures may be normed along isogenies.

\begin{rec}[{\cite[\S 14]{AHS}}]
  Suppose we are given an isogeny $G \to G'$ of formal groups or elliptic curves.
  Given a line bundle $\L$ over $G$, there is a line bundle $N(\L)$ over $G'$, called the \textit{norm} of $\L$.
  Moreover, given a $\Theta^1$-structure $s$ on $\L$, there is an associated $\Theta^1$-structure $N(s)$ on $N(\L)$, called the \textit{norm} of $s$.
\end{rec}

\begin{rmk}
  Given an isogeny of formal groups $\G \to \wt{\G}$, there is a natural isomorphism of line bundles $N(\OO_\G (1)) \cong \OO_{\wt{\G}} (1)$.
\end{rmk}

\begin{rmk}
  Suppose that we are given an elliptic curve $E$ with a point $P$ of exact order $n$, i.e. a level-$\Gamma_1 (n)$ structure, and an isogeny $E \to \wt{E}$ of degree $p$ coprime to $n$.
  Then the image $\wt{P}$ of $P$ in $\wt{E}$ is again a point of exact order $n$, and there is a natural isomorphism of line bundles $N (\OO_E ([0] - [P])) \cong \OO_{\wt{E}} ([0] - [\wt{P}])$.
\end{rmk}

Following \cite{AHS}, we may now use this language to rephrase the Ando criterion for Lubin--Tate theory.
We begin with some setup.


\begin{rec}
  Let $k$ denote a perfect field of characteristic $p > 0$ and let $\G$ denote a formal group of finite height over $k$, so that we have an associated Lubin--Tate theory $E(k,\G)$.
  Then there are two ring maps $i, \psi : \pi_0 E(k, \G) \to E(k,\G)^0 (\BC_p) / I_{tr}$.
  The first, $i$, is induced by the projection $\BC_p \to \ast$. The second $\psi$, is the total power operation.
  Over the ring $E(k,\G)^0 (\BC_p) /I_{tr}$, there is a degree $p$ isogeny
  \[i^* \Guniv \to \psi^* \Guniv \]
  induced by the total power operation on $E(k, \G)^0 (\CP^\infty)$.
\end{rec}

  Given a $\Theta^1$-structure $s$ on $\OO_{\Guniv} (1)$, there are therefore two naturally induced $\Theta^1$-structures on $\OO_{\psi^* \Guniv} (1) \cong N(\OO_{i^* \Guniv} (1))$: the pullback $\psi^* (s)$ and the norm $N(i^* (s))$.


\begin{dfn}
  We say that a $\Theta^1$-structure $s$ on $\Guniv$ satisfies the Ando criterion if $\psi^* (s) = N(i^* (s))$.
\end{dfn}

It follows from \cite[\S 5]{AHS} that this is compatible with our previous definition of the Ando criterion:

\begin{prop} \label{prop:Ando-theta}
  A complex orientation of $E(k, \G)$ satisfies the Ando criterion if and only if the associated $\Theta^1$-structure on $\OO_{\wt{\G}} (1)$ satisfies the Ando criterion.
\end{prop}


We are now able to prove the main theorem of this section:

\begin{thm} \label{thm:TMF}
  The Hirzebruch level $n$ elliptic genus $\MU \to \tmf_1 (n) \to \TMF_1 (n)$ satisfies the Ando criterion.
  As a consequence of \Cref{thm:main-obst} and \Cref{rmk:tmf1n-nice}, it lifts uniquely up to homotopy to an $\Einf$-complex orientation
  \[\MU \to \TMF_1 (n).\]
\end{thm}

\begin{proof}
  Choose, for each $p$ not dividing $n$, $(E, P \in E[n] (k)) \in \M_1 (n)(k)$ with $E$ supersingular and $k$ a perfect field of characteristic $p$. By \Cref{prop:reduction-LT}, it suffices to show that the induced complex orientation of $E(k,\widehat{E})$ satisfies the Ando criterion.

  By the Serre--Tate theorem \cite[\S 1]{ST}, $\pi_0 E(k, \widehat{E})$ is the universal deformation ring of of $E$.
  We let $E^{\univ}$ denote the universal deformation of $E$ over $\pi_0 E(k, \widehat{E})$, and let $P^\univ \in E^{\univ}[n](\pi_0 E(k,\G))$ denote the unique lift of $P$.
  The associated formal group $\widehat{E^{\univ}}$ is a universal deformation of $\widehat{E}$.

  Applying \Cref{prop:Ando-theta}, we must show that the $\Theta^1$-structure $s$ on $\OO_{\widehat{E^\univ}} (1)$ corresponding to the level $n$ elliptic genus satisfies the Ando criterion, i.e. that $\psi^* (s) = N(i^* (s))$.
%
  By definition of the level $n$ elliptic genus, there is a degree $n$ \'etale isogeny $q: (E^{\univ})' \to E^\univ$ and a $\Theta^1$-structure $\overline{s}$ on $q^* \OO_{E^{\univ}} ([0] - [P^{\univ}])$ which induces $s$.
  By the Serre--Tate theorem, the isogeny of formal groups
  \[i^* \widehat{E^{\univ}} \to \psi^* \widehat{E^{\univ}}\]
  over $E(k, \widehat{E})^0 (\BC_p)/ I_{tr}$ lifts to a diagram of isogenies of elliptic curves
  \begin{center}
    \begin{tikzcd}
      i^* (E^{\univ})' \ar[r] \ar[d, "q_i"] & \psi^* (E^\univ)' \ar[d,"q_\psi"] \\
      i^* E^\univ \ar[r] & \psi^* E^\univ.
    \end{tikzcd}
  \end{center}

  Let $\overline{P^{\univ}}$ denote the image of $i^*(P^\univ )$ in $\psi^* E^\univ$.
  From the above diagram, we obtain $\Theta^1$-structures $\psi^* (\overline{s})$ and $N(i^* (\overline{s}))$ on
  \[(q_\psi )^* \OO_{\psi^* E^\univ} ([0]-[\overline{P^\univ}]) \cong N ((q_i )^* \OO_{i^* E^\univ}([0]-[i^*(P^\univ)])).\]
  As these induce $\psi^* (s)$ and $N( i^* (s))$, it suffices to show that $\psi^* (\overline{s}) = N(i^* (\overline{s}))$.
  But this follows immediately from the uniqueness of $\Theta^1$-structures over elliptic curves proven in \Cref{lem:ell-theta-unique}.
\end{proof}

\subsection{Lift to $\tmf_1 (n)$}\label{sec:lift}

In this section, we will complete the proof of \Cref{thm:main-elliptic} by proving the following two lemmas:

\begin{lem} \label{lem:lift}
  Suppose that we are given a complex orientation $\MU \to \Tmf_1 (n)$ with the property that the composite
  \[\MU \to \Tmf_1 (n) \to \TMF_1 (n)\]
  lifts to an $\Einf$-ring map.
  Then this complex orientation lifts uniquely to an $\Einf$-complex orientation $\MU \to \Tmf_1 (n)$.
\end{lem}

\begin{lem} \label{lem:tT}
  Any $\Einf$-complex orientation of $\Tmf_1 (n)$ lifts uniquely to an $\Einf$-complex orientation of $\tmf_1 (n)$.
\end{lem}

\begin{proof}[Proof of \Cref{lem:lift}]
  The uniqueness will follow from \Cref{thm:unique} and \Cref{exm:unique-tmf} once we know that a complex orientation of $\Tmf_1 (n)$ is determined by the induced complex orientation of $\TMF_1 (n)$.
  This follows from the fact that the map $\pi_* \Tmf_1 (n) \to \pi_* \TMF_1 (n)$ is injective in even degrees by \cite[Proposition 2.5]{MeierTop} and the descent spectral sequence.

  It therefore suffices to show that the $\Einf$-map $\MU \to \TMF_1 (n)$ lifts to an $\Einf$-map $\MU \to \Tmf_1 (n)$.
  We begin with the pullback square of $\Einf$-rings coming from \cite{HLtmf}:
  \begin{center}
    \begin{tikzcd}
      \Tmf_1 (n) \ar[r] \ar[d] & \TMF_1 (n) \ar[d] \\
      K^{\cusp}_1 (n) \ar[r] & \Delta^{-1} K^{\cusp}_1 (n)
    \end{tikzcd}
  \end{center}
  This square satisfies the following properties:
  \begin{enumerate}
    \item The $\Einf$-rings $K^{\cusp}_1 (n)$ and $\Delta^{-1} K^{\cusp} _1 (n)$ are Landweber exact and of height $\leq 1$.
    \item The induced map $\pi_* K^{\cusp}_1 (n)/p \to \pi_* \Delta^{-1} K^{\cusp} _1 (n)/p$ is injective for all $p$.
  \end{enumerate}
  It follows that to construct an $\Einf$-lifting $\MU \to \Tmf_1 (n)$, it suffices to lift the composite
  \[\MU \to \TMF_1 (n) \to \Delta^{-1} K^{\cusp} _1 (n)\]
  to an $\Einf$-map
  \[\MU \to K^{\cusp} _1 (n).\]
  By the uniqueness in \Cref{thm:main-obst} and (1) above, it suffices to lift the complex orientation and verify that it satisfies the Ando criterion.
  But we are given a lift of the complex orientation by assumption, and it follows from \Cref{lem:BCp-inj} and (2) above that it satisfies the Ando criterion.
\end{proof}

\begin{proof}[Proof of \Cref{lem:tT}]
  For this, we use orientation theory \cite{ABGHR}.
  We have the sequence of maps
  \[\bu \xrightarrow{J} \bgl_1 (\Ss) \to \bgl_1 (\tmf_1 (n)) \to \bgl_1 (\Tmf_1 (n)),\]
  and $\Einf$-complex orientations of $\tmf_1$ and $\Tmf_1 (n)$ correspond to nullhomotopies of the respective composites.
  Since $\bu$ is $2$-connective, we may as well replace all occurrences of $\bgl_1$ with $\bsl_1 \coloneqq \tau_{\geq 2} \bgl_1$.

  Now, it follows from the definition of $\tmf_1 (n)$ \cite{Meierleveln} that the composite
  \[\bsl_1 (\tmf_1 (n)) \to \bsl_1 (\Tmf_1 (n)) \to \tau_{\geq 3} \bsl_1 (\Tmf_1 (n))\]
  is an equivalence, so that there is a splitting
  \[\bsl_1 (\Tmf_1 (n)) \simeq \bsl_1 (\tmf_1 (n)) \oplus \Sigma^2 \pi_1 \Tmf_1 (n),\]
  from which the lemma follows.
\end{proof}

\begin{proof}[Proof of \Cref{thm:main-elliptic}]
  Combine \Cref{thm:TMF} with \Cref{lem:lift} and \Cref{lem:tT}.
\end{proof}

\appendix
\section{$K(n)$-localizations of Landweber exact ring spectra} \label{sec:app}
In this appendix, we prove the following lemma, which the author was unable to find a reference for in the literature:

\begin{lem} \label{lem:Kn-Land}
  Let $E$ denote a Landweber exact ring spectrum whose homotopy is concentrated in even degrees.
  Then $L_{K(n)} E$ is also a Landweber exact ring spectrum whose homotopy is concentrated in even degrees.
\end{lem}

Before we proceed to the proof of \Cref{lem:Kn-Land}, we need a lemma from commutative algebra.

\begin{lem} \label{lem:reg-comp}
  Let $R_*$ denote a graded commutative ring and let $x_1, \dots, x_n \in R_*$ denote a regular sequence of homogeneous elements.
  Then the sequence $x_1, \dots, x_n$ remains regular in the completion $(R_*) ^{\wedge} _{(x_1, \dots x_n)}$.
\end{lem}

\begin{proof}
  Let $I = (x_1, \dots x_n)$.
  We claim that there are short exact sequences
  \begin{align}\label{eq:ses}
    0 \to R_* / I^{k-1} \xrightarrow{x_1} R_* / I^{k} \to R_* / (I^k + (x_1)) \to 0 
  \end{align}
  for all $k \geq 1$.
  Supposing this for the moment, we find by taking the limit that
  \[ 0 \to (R_*)^{\wedge} _I \xrightarrow{x_1} (R_*)^{\wedge} _I \to (R_* / x_1)^{\wedge} _I \to 0 \]
  is short exact.
  In particular, we find that $x_1$ is a regular element in $(R_*)^{\wedge} _I$ and that $(R_*)^{\wedge} _I / x_1 \cong (R_* / x_1)^{\wedge} _I$.
  Taking $R_* / x_1$ as our new ring and the image of $x_2, \dots, x_n$ in $R_* / x_1$ as our exact sequence, we may conclude by induction on the length of our regular sequence.
  
  It remains to establish (\ref{eq:ses}).
  It is sufficient to prove that (\ref{eq:ses}) is exact on the associated graded of the $I$-adic filtration where $x_1$ is considered as a map of $I$-adic filtration degree $1$.
  Since $x_1, \dots, x_n$ is regular, this associated graded may be identified with the sequence
	\begin{center}
		    \begin{tikzcd}[column sep=small]
          0 \arrow[r] & (R_* / I) [\overline{x}_1, \dots \overline{x}_n] / (\overline{x}_1, \dots, \overline{x}_n)^{k-1} \arrow[rr, "\overline{x}_1"] & \arrow[d, phantom, ""{coordinate, name=Z}] & (R_* / I) [\overline{x}_1, \dots \overline{x}_n] / (\overline{x}_1, \dots, \overline{x}_n)^{k} \arrow[dll, ""'{pos=1}, rounded corners, to path={ -- ([xshift=2ex]\tikztostart.east) |- (Z) [near end]\tikztonodes -| ([xshift=-2ex]\tikztotarget.west) --(\tikztotarget)}] & \\
          & (R_* / I) [\overline{x}_2, \dots \overline{x}_{n}] / (\overline{x}_2, \dots, \overline{x}_{n})^{k} \arrow[rr] & \vphantom{a} & 0, & \vphantom{a}
	\end{tikzcd}
	\end{center}
%
  which is easily verified to be short exact.
\end{proof}

\begin{proof} [Proof of \Cref{lem:Kn-Land}]
  By abuse of notation, we let $v_i \in \pi_* E$ inductively denote an arbitrary lift of the class $v_i \in \pi_* E / (p, \dots, v_{i-1})$.
  Given a positive integer $k$, let $I_k = (p, v_1, \dots, v_{k-1})$ and let $I_0 = (0)$.
  It follows from \cite[Proposition 7.10]{HoveyStrickland} that $\pi_* L_{K(n)} E \cong (v_n ^{-1} \pi)^{\wedge} _{I_n}$.
  In particular, $\pi_{2*-1} L_{K(n)} E = 0$.
  It is clear that the operation of inverting $v_n ^{-1}$ preserves Landweber exactness, so that it suffices to prove that completion with respect to $I_n$ does as well.
  But this follows immediately from \Cref{lem:reg-comp}.
\end{proof}

%
%
%

\bibliographystyle{alpha}
\bibliography{bibliography}

\end{document}